\documentclass[12pt,reqno]{amsart}
\usepackage[left=80pt,right=80pt]{geometry}
\usepackage[usenames]{color}
\usepackage{amsmath}
\usepackage{amssymb}
\usepackage{amsthm}
\usepackage{enumitem}
\usepackage{float}

\usepackage{graphicx}
\usepackage{tikz}
\usepackage{subcaption}
\usepackage{mathdots}

\usepackage{hyperref,url}
\hypersetup{
	colorlinks=true,
  linkcolor=black,          
  citecolor=black,         
  filecolor=black,      
  urlcolor=black           
}
\newtheorem{prop}{Proposition}
\newtheorem{lemma}[prop]{Lemma}
\newtheorem{theorem}[prop]{Theorem}

\newtheorem{corollary}[prop]{Corollary}
\theoremstyle{definition}
\newtheorem{definition}[prop]{Definition}
\newtheorem{remark}[prop]{Remark}
\newtheorem{example}[prop]{Example}

\newtheorem*{conclusion*}{Conclusion}

\newcommand{\R}{\mathbb{R}}

\newcommand{\W}{\mathcal{W}}
\newcommand{\C}{\mathcal{C}}
\newcommand{\seqnum}[1]{\href{https://oeis.org/#1}{\rm \underline{#1}}}
\allowdisplaybreaks

\newcommand{\rc}{\textnormal{rc}}

\begin{document}
\tikzset{mystyle/.style={matrix of nodes,
        nodes in empty cells,
        row 1/.style={nodes={draw=none}},
        row sep=-\pgflinewidth,
        column sep=-\pgflinewidth,
        nodes={draw,minimum width=1cm,minimum height=1cm,anchor=center}}}
\tikzset{mystyleb/.style={matrix of nodes,
        nodes in empty cells,
        row sep=-\pgflinewidth,
        column sep=-\pgflinewidth,
        nodes={draw,minimum width=1cm,minimum height=1cm,anchor=center}}}

\title{Counting $r\times s$ rectangles in (Catalan) words}

\author[SELA FRIED]{Sela Fried$^{\dagger}$}
\thanks{$^{\dagger}$ Department of Computer Science, Israel Academic College,
52275 Ramat Gan, Israel.
\\
\href{mailto:friedsela@gmail.com}{\tt friedsela@gmail.com}}
\author[TOUFIK MANSOUR]{Toufik Mansour$^{\sharp}$}
\thanks{$^{\sharp}$ Department of Mathematics, University of Haifa, 3103301 Haifa, Israel.\\
\href{mailto:tmansour@univ.haifa.ac.il}{\tt tmansour@univ.haifa.ac.il}}

\begin{abstract}
Generalizing previous results, we introduce and study a new statistic on  words, that we call rectangle capacity. For two fixed positive integers $r$ and $s$, this statistic counts the number of occurrences of a rectangle of size $r\times s$ in the bargraph representation of a word. We find the bivariate generating function for the distribution on words of the number of $r\times s$ rectangles and the generating function for their total number over all words. We also obtain the analog results for Catalan words.
\bigskip


\noindent
\textbf{Math.~Subj.~Class.:} 05A05, 05A15.
\end{abstract}

\maketitle

\section{Introduction}
Let $n$ and $k$ be two positive integers and set $[k] = \{1,2,\ldots,k\}$. A word over the alphabet $[k]$ of length $n$ is any element of the set $[k]^n$. Words have a visual representation in terms of bargraphs (see Figure \ref{fig;1}), giving rise to many natural statistics on words, such as their water capacity, number of lit cells, and  perimeter (e.g., \cite{BBK}, \cite{ABBKM}, and \cite{BBKM}. See \cite{MS2} for a comprehensive review of works studying enumeration problems in bargraphs).

Generalizing previous results, we introduce and study a new statistic, that we call {\it rectangle capacity}. For two fixed positive integers $r$ and $s$, this statistic counts the number of occurrences of a rectangle of size $r\times s$ in the bargraph representation of a word (see Figure \ref{fig;2}). Actually, rectangle capacity is a special case of a much more general statistic, defined in \cite{MS1}: For two bargraphs $B$ and $C$, a vertex $(x,y)$ of $B$ is said to be a $C$-vertex if $C$ lies entirely in $B$, when positioned starting at $(x,y)$. The authors of \cite{MS1} studied the number of $C$-vertices in several special cases. Two of which correspond to our statistic of rectangle capacity with $1\times s$ and $r\times 1$. Rectangle capacity has also been studied in the context of set partitions by \cite{C}, for $r\times s= 1\times 2$ and by \cite{AB} for $r\times s= 2\times 2$. Notice that the rectangle capacity of $1\times 1$ rectangles coincides with the area of the bargraph, a statistic that has been studied in several cases as well (e.g., \cite{6, MS2}). 

\begin{figure}[H]
\begin{center}
\centering
\scalebox{.40}{
\begin{tikzpicture}
\foreach \y in {1,2,3}
{\node[line width=1mm] at (0,\y) [rectangle,draw=black!35, minimum size=1cm]  {};}

\foreach \y in {1,2,3,4}
{\node[line width=1mm] at (1,\y) [rectangle,draw=black!35, minimum size=1cm]  {};}

\foreach \y in {1,2,3,4,5}
{\node[line width=1mm] at (2,\y) [rectangle,draw=black!35, minimum size=1cm]  {};}

\foreach \y in {1}
{\node[line width=1mm] at (3,\y) [rectangle,draw=black!35, minimum size=1cm]  {};}

\foreach \y in {1,2,3}
{\node[line width=1mm] at (4,\y) [rectangle,draw=black!35, minimum size=1cm]  {};}

\foreach \y in {1,2,3,4}
{\node[line width=1mm] at (5,\y) [rectangle,draw=black!35, minimum size=1cm]  {};}
\end{tikzpicture}}
\end{center}
\caption{The bargraph representation of the word $345134$}
\label{fig;1}
\end{figure}
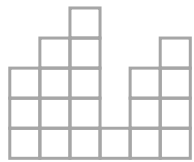%

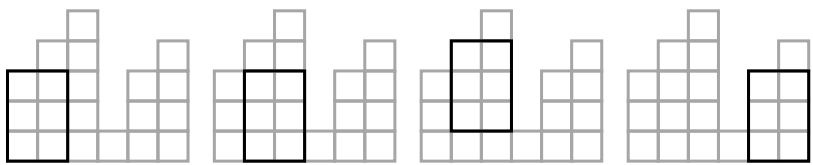
\begin{figure}[H]
\centering
\scalebox{.40}{
\begin{tikzpicture}
\foreach \y in {1,2,3}
{\node[line width=1mm] at (0,\y) [rectangle,draw=black!35, minimum size=1cm]  {};}

\foreach \y in {1,2,3,4}
{\node[line width=1mm] at (1,\y) [rectangle,draw=black!35, minimum size=1cm]  {};}

\foreach \y in {1,2,3,4,5}
{\node[line width=1mm] at (2,\y) [rectangle,draw=black!35, minimum size=1cm]  {};}

\foreach \y in {1}
{\node[line width=1mm] at (3,\y) [rectangle,draw=black!35, minimum size=1cm]  {};}

\foreach \y in {1,2,3}
{\node[line width=1mm] at (4,\y) [rectangle,draw=black!35, minimum size=1cm]  {};}

\foreach \y in {1,2,3,4}
{\node[line width=1mm] at (5,\y) [rectangle,draw=black!35, minimum size=1cm]  {};}
\node[line width=1mm] at (0.5,2) [rectangle,draw=black!35, minimum width=2cm, minimum height = 3cm, color=black]  {};
\end{tikzpicture}
\hspace{0.5cm}
\begin{tikzpicture}
\foreach \y in {1,2,3}
{\node[line width=1mm] at (0,\y) [rectangle,draw=black!35, minimum size=1cm]  {};}

\foreach \y in {1,2,3,4}
{\node[line width=1mm] at (1,\y) [rectangle,draw=black!35, minimum size=1cm]  {};}

\foreach \y in {1,2,3,4,5}
{\node[line width=1mm] at (2,\y) [rectangle,draw=black!35, minimum size=1cm]  {};}

\foreach \y in {1}
{\node[line width=1mm] at (3,\y) [rectangle,draw=black!35, minimum size=1cm]  {};}

\foreach \y in {1,2,3}
{\node[line width=1mm] at (4,\y) [rectangle,draw=black!35, minimum size=1cm]  {};}

\foreach \y in {1,2,3,4}
{\node[line width=1mm] at (5,\y) [rectangle,draw=black!35, minimum size=1cm]  {};}
\node[line width=1mm] at (1.5,2) [rectangle,draw=black!35, minimum width=2cm, minimum height = 3cm, color=black]  {};
\end{tikzpicture}
\hspace{0.5cm}
\begin{tikzpicture}
\foreach \y in {1,2,3}
{\node[line width=1mm] at (0,\y) [rectangle,draw=black!35, minimum size=1cm]  {};}

\foreach \y in {1,2,3,4}
{\node[line width=1mm] at (1,\y) [rectangle,draw=black!35, minimum size=1cm]  {};}

\foreach \y in {1,2,3,4,5}
{\node[line width=1mm] at (2,\y) [rectangle,draw=black!35, minimum size=1cm]  {};}

\foreach \y in {1}
{\node[line width=1mm] at (3,\y) [rectangle,draw=black!35, minimum size=1cm]  {};}

\foreach \y in {1,2,3}
{\node[line width=1mm] at (4,\y) [rectangle,draw=black!35, minimum size=1cm]  {};}

\foreach \y in {1,2,3,4}
{\node[line width=1mm] at (5,\y) [rectangle,draw=black!35, minimum size=1cm]  {};}
\node[line width=1mm] at (1.5,3) [rectangle,draw=black!35, minimum width=2cm, minimum height = 3cm, color=black]  {};
\end{tikzpicture}
\hspace{0.5cm}
\begin{tikzpicture}
\foreach \y in {1,2,3}
{\node[line width=1mm] at (0,\y) [rectangle,draw=black!35, minimum size=1cm]  {};}

\foreach \y in {1,2,3,4}
{\node[line width=1mm] at (1,\y) [rectangle,draw=black!35, minimum size=1cm]  {};}

\foreach \y in {1,2,3,4,5}
{\node[line width=1mm] at (2,\y) [rectangle,draw=black!35, minimum size=1cm]  {};}

\foreach \y in {1}
{\node[line width=1mm] at (3,\y) [rectangle,draw=black!35, minimum size=1cm]  {};}

\foreach \y in {1,2,3}
{\node[line width=1mm] at (4,\y) [rectangle,draw=black!35, minimum size=1cm]  {};}

\foreach \y in {1,2,3,4}
{\node[line width=1mm] at (5,\y) [rectangle,draw=black!35, minimum size=1cm]  {};}
\node[line width=1mm] at (4.5,2) [rectangle,draw=black!35, minimum width=2cm, minimum height = 3cm, color=black]  {};
\end{tikzpicture}}
\caption{There are four occurrences of a rectangle of size $3\times 2$ in the word $345134$.}\label{fig;2}
\end{figure}

Catalan words of length $n$ are words $w_{1}\cdots w_{n}$ satisfying $w_{1}=1$ and $w_{i+1}\leq w_{i}+1$, for every $1\leq i\leq n$. Being counted by the Catalan numbers, Catalan words are well studied and still continue to attract attention from the combinatorics community (e.g., \cite{BAR1, BAR3, BAR2, SHa}). Nevertheless, research on the distribution of statistics on Catalan words in terms of their representation as bargraphs, seems to have begun only recently with \cite{6, 9, 10}. 

In this work we find the bivariate generating function for the distribution on words (resp.\ Catalan words) of the number of $r\times s$ rectangles and the generating function for their total number over all words (resp.\ Catalan words).

Before we begin, let $r,s$, and $k$ be three positive integers and let $n$ be a nonnegative integer, to be used throughout this work. If $m$ is a positive integer, we denote by $[m]$ the set $\{1,2,\ldots,m\}$. 

\section{Main results - Words}

Let $\W_{n,k} = [k]^n$ stand for the set of all words over $[k]$ of length $n$ and denote by $a_{n,k}=a_{n,k}(t)$ the distribution on $\W_{n,k}$ of the number of $r\times s$ rectangles. Let $A_k(x,t)$ denote the generating function of the numbers $a_{n,k}$. We distinguish between two cases, namely $r=1$ and $r\geq 2$. As was already mentioned in the introduction, the case $1\times s$ has already been settled by \cite[Theorem 2.2]{MS1}. For completeness, we state, without proof, the result also in this case. Notice that our formulation is somewhat simpler, since we do not consider the dependency on the number of cells of the bargraphs.

\subsection{\texorpdfstring{$1\times s$}{} rectangles}

\begin{theorem}
We have
\[
A_1(x,t) =
\frac{1-\sum_{n=1}^{s-1}(t-1)x^{n}}{1-tx}\]
and, for $k\geq 2$,
\[A_{k}(x,t)=\alpha_{k}(x,t)+\frac{1/t^{s}x}
{1-tx\alpha_{k-1}(tx,t)-\frac{1/t^{s}}
{1-t^{2}x\alpha_{k-2}(t^{2}x,t)-\frac{1/t^{s}}
{\ddots1-t^{k-2}x\alpha_2(t^{k-2}x,t)-\frac{\overset{\vdots}
{1/t^{s}}}{1-t^{k-1}xA_1(t^{k-1}x,t)}
\iddots}}},\]
where
\[\alpha_{k}(x,t)=\frac{1-(kx)^{s}}{1-kx}-\frac{1-(ktx)^{s}}{t^{s-1}(1-ktx)}-\frac{1}{t^{s}x}.\]
\end{theorem}

\begin{remark}
The total number of $1\times s$ words over all words is established in Theorem \ref{eq;hs1}. See also Table \ref{tablea1}.
\end{remark}

\subsection{\texorpdfstring{$r\times s$}{} rectangles, where \texorpdfstring{$r\geq 2$}{}}

For $m\in [r]$, we denote by $\W_{n,k}^{\geq m}$ the subset of $\W_{n,k}$ consisting of those words $w=w_1\cdots w_n$ such that $w_i\geq m$, for every $i\in [n]$. Let $a^{\geq m}_{n,k}$ denote the restriction of $a_{n,k}$ to $\W_{n,k}^{\geq m}$ and let $A_k^{\geq m}(x,t)$ denote the generating function of the numbers $a^{\geq m}_{n,k}$.

\begin{lemma}
Let $k\geq r-1$. Then
\begin{equation}\label{eq;842}
A_k^{\geq r-1}(x,t) =
\begin{cases}
 \frac{1}{1-x}&\textnormal{if } k = r-1 \\
 \frac{\alpha(x,t)A_{k-1}^{\geq r-1}(tx,t) +\beta_k(x,t)}{\gamma(x,t)A_{k-1}^{\geq r-1}(tx,t) +\delta_k(x,t)}&  \textnormal{otherwise},
\end{cases}
\end{equation} where $\alpha(x,t) = 1/t^{s-1}, \gamma(x,t) = -x/t^{s-1}$, and
\begin{align}
\beta_k(x,t)&=\frac{1-((k-r+2)x)^{s}}{1-(k-r+2)x}-\frac{1-((k-r+1)tx)^{s}}{t^{s-1}(1-(k-r+1)tx)}\nonumber\\&-\frac{x}{1-(k-r+2)x}\left(\frac{1-((k-r+1)x)^{s-1}}{1-(k-r+1)x}-((k-r+2)x)^{s-1}\frac{1-\left(\frac{k-r+1}{k-r+2}\right)^{s-1}}{1-\frac{k-r+1}{k-r+2}}\right),\nonumber\\
\delta_k(x,t) &=1-\frac{x(1-((k-r+1)x)^{s-1})}{1-(k-r+1)x}+\frac{x(1-((k-r+1)tx)^{s-1})}{t^{s-1}(1-(k-r+1)tx)}.\nonumber
\end{align}
\end{lemma}

\begin{proof}
If $n\leq s-1$, no word of length $n$ can contain an $r\times s$ rectangle. Since there are $(k-r+2)^n$ such words, we have $a_{n,k}^{\geq r-1} = t^0 (k-r+2)^n = (k-r+2)^n$. Assume now that $n\geq s$.
If $k=r-1$, then $\W_{n,r-1}^{\geq r-1}$ consists of a single word $(r-1)\cdots (r-1)$ of length $n$, containing no  $r\times s$ rectangles. Thus, $a_{n,r-1}^{\geq r-1} = 1$ and it follows that $A_{r-1}^{\geq r-1}(x,t)=1/(1-x)$. Assume that $k\geq r$. We have
\begin{align}
a_{n,k}^{\geq r-1}&=a_{n,k}^{\geq r}+\sum_{i=1}^{n}a_{i-1,k}^{\geq r}a_{n-i,k}^{\geq r-1}\nonumber\\&=t^{n-s+1}a_{n,k-1}^{\geq r-1}+\sum_{i=1}^{s-1}a_{i-1,k-1}^{\geq r-1}a_{n-i,k}^{\geq r-1}+\sum_{i=s}^{n}t^{i-s}a_{i-1,k-1}^{\geq r-1}a_{n-i,k}^{\geq r-1}.  \nonumber
\end{align}
Multiplying both sides of this equation by $x^n$, summing over $n\geq s$ and adding $\sum_{n=0}^{s-1}a_{n,k}^{\geq r-1}x^{n}$ to both sides, with some algebra we obtain the equation
\begin{align}
&A_{k}^{\geq r-1}(x,t)\bigg(-\frac{x}{t^{s-1}}A_{k-1}^{\geq r-1}(tx,t)+1-\frac{x(1-((k-r+1)x)^{s-1})}{1-(k-r+1)x}+\nonumber\\&  \hspace{10cm} \frac{x(1-((k-r+1)tx)^{s-1})}{t^{s-1}(1-(k-r+1)tx)}\bigg)=\nonumber  \\&\frac{1}{t^{s-1}}A_{k-1}^{\geq r-1}(tx,t)+\frac{1-((k-r+2)x)^{s}}{1-(k-r+2)x}-\frac{1}{t^{s-1}}\frac{1-((k-r+1)tx)^{s}}{1-(k-r+1)tx}-\nonumber\\&\hspace{1cm}\frac{x}{1-(k-r+2)x}\left(\frac{1-((k-r+1)x)^{s-1}}{1-(k-r+1)x}-((k-r+2)x)^{s-1}\frac{1-\left(\frac{k-r+1}{k-r+2}\right)^{s-1}}{1-\frac{k-r+1}{k-r+2}}\right),\nonumber
\end{align}
from which the statement immediately follows.
\end{proof}

\begin{theorem}\label{thm;a0}
For $k\geq r$, we have
\[
A_k(x,t) =
\frac{\bar{\alpha}(x,t)A_{k-1}^{\geq r-1}(tx,t) +\bar{\beta}_k(x,t)}{\bar{\gamma}(x,t)A_{k-1}^{\geq r-1}(tx,t) +\bar{\delta}_k(x,t)},
\]
where $\bar{\alpha}(x,t) = 1/t^{s-1}, \bar{\gamma}(x,t) = -(r-1)x/t^{s-1}$, and
\begin{align}
\bar{\beta}_k(x,t)&=\frac{1-(kx)^{s}}{1-kx}-\frac{1-((k-r+1)tx)^{s}}{t^{s-1}(1-(k-r+1)tx)}\nonumber\\&-\frac{(r-1)x}{1-kx}\left(\frac{1-((k-r+1)x)^{s-1}}{1-(k-r+1)x}-(kx)^{s-1}\frac{1-\left(\frac{k-r+1}{k}\right)^{s-1}}{1-\frac{k-r+1}{k}}\right),\nonumber\\
\bar{\delta}_k(x,t) &=1-\frac{(r-1)x(1-((k-r+1)x)^{s-1})}{1-(k-r+1)x}+\frac{(r-1)x(1-((k-r+1)tx)^{s-1})}{t^{s-1}(1-(k-r+1)tx)}.\nonumber
\end{align}

\begin{proof}
If $n\leq s-1$, no word of length $n$ can contain a $1\times s$ rectangle. Since there are $k^n$ such words, we have $a_{n,k} = t^0 k^n = k^n$. Assume now that $n\geq s$. We have
\begin{align*}
a_{n,k}&=t^{n-s+1}a_{n,k}^{\geq r}+(r-1)\sum_{i=1}^{n}a_{i-1,k}^{\geq r}a_{n-i,k}\\&=t^{n-s+1}a_{n,k-1}^{\geq r-1}+(r-1)\sum_{i=1}^{s-1}a_{i-1,k-1}^{\geq r-1}a_{n-i,k}+(r-1)\sum_{i=s}^{n}t^{i-s}a_{i-1,k-1}^{\geq r-1}a_{n-i,k}.
\end{align*}
Multiplying both sides of this equation by $x^n$, summing over $n\geq s$ and adding $\sum_{n=0}^{s-1} a_{n,k}x^n$ to both sides, with some algebra we obtain the equation

\begin{align}
&A_{k}(x,t)\bigg(-\frac{(r-1)x}{t^{s-1}}A_{k-1}^{\geq r-1}(tx,x)+1-\frac{(r-1)x(1-((k-r+1)x)^{s-1})}{1-(k-r+1)x}+\nonumber\\ & \hspace{8.5cm} \frac{(r-1)x(1-((k-r+1)tx)^{s-1})}{t^{s-1}(1-(k-r+1)tx)}\bigg)=\nonumber\\ &\frac{1}{t^{s-1}}A_{k-1}^{\geq r-1}(tx,t)+\frac{1-(kx)^{s}}{1-kx}-\frac{1-((k-r+1)tx)^{s}}{t^{s-1}(1-(k-r+1)tx)}-\nonumber\\&\hspace{4cm}\frac{(r-1)x}{1-kx}\left(\frac{1-((k-r+1)x)^{s-1}}{1-(k-r+1)x}-(kx)^{s-1}\frac{1-\left(\frac{k-r+1}{k}\right)^{s-1}}{1-\frac{k-r+1}{k}}\right),\nonumber
\end{align} from which the statement immediately follows.
\end{proof}
\end{theorem}

\begin{remark}\label{rem;abcd1}
In principle, it is possible, but rather tedious, to obtain the total number of $r\times s$ words over all words of length $n$, inductively by differentiating \eqref{eq;842} with respect to $t$ and then substituting $t=1$. In the following, we present a more elegant way, that works for every $r\in[k]$. We shall make use of the following identity:
\begin{equation}\label{eq;1122}
\sum_{x_1,\ldots,x_s\in[k]}\max\left\{ \min\left\{ x_1,\ldots,x_s\right\}-r+1,0\right\} =\sum_{j=1}^{k-r+1}j^s.
\end{equation} This identity follows from the following identity, which is not hard to prove, e.g., with the inclusion-exclusion principle:
\[\sum_{x_{1},\ldots,x_r\in[s]}\min\left\{ x_{1},\ldots,x_r\right\} =\sum_{j=1}^sj^r. \]
\end{remark}

\begin{theorem}\label{eq;hs1}
Assume that $r\in [k]$ and let $f_{n,k}$ denote the total number of $r\times s$ words over all words of length $n$. Let $F_k(x)$ denote the generating function of the numbers $f_{n,k}$. Then
\[F_k(x)=\frac{x^s\sum_{j=1}^{k-r+1}j^s}{(1-kx)^{2}}.\] Thus, for $n\geq s$, we have
\[f_{n,k} =
(n-s+1)k^{n-s}\sum_{j=1}^{k-r+1}j^s.\]
\end{theorem}

\begin{proof}
For $w_1\cdots w_{s-1}\in \W_{s-1,k}$, we denote by $\W_{n,k}^{w_1\cdots w_{s-1}}$ the subset of $\W_{n,k}$ consisting of those words $\bar{w}_1\cdots \bar{w}_n\in\W_{n,k}$ such that $\bar{w}_i = w_i$, for every $i\in[s-1]$. Let $a_{n,k}^{w_1\cdots w_{s-1}}$ be the restriction of $a_{n,k}$ to $\W_{n,k}^{w_1\cdots w_{s-1}}$ and let $A_k^{w_1\cdots w_{s-1}}(x,t)$ denote the generating function of the numbers $a^{w_1\cdots w_{s-1}}_{n,k}$. We have
\begin{equation}\label{eq;1a}
A(x,t)^{w_1\cdots w_{s-1}}=x^{s-1}+x\sum_{w_0\in[k]}t^{\max\left\{ \min\left\{ w_0,\ldots,w_{s-1}\right\} -r+1,0\right\} }A(x,t)^{w_0\cdots w_{s-2}}.
\end{equation}
Differentiating \eqref{eq;1a} with respect to $t$ and substituting $t=1$, we obtain
\begin{align}
\frac{\partial}{\partial t}A(x,t)^{w_1\cdots w_{s-1}}_{|(x,t)=(x,1)}&=x\sum_{w_0\in[k]}\max\left\{ \min\left\{ w_0,\ldots,w_{s-1}\right\} -r+1,0\right\} A(x,1)^{w_0\cdots w_{s-2}}\nonumber\\&+x\frac{\partial}{\partial t}A(x,t)^{w_0\cdots w_{s-2}}_{|(x,t)=(x,1)}.\label{eq;1b}
\end{align} Summing \eqref{eq;1b} over $w_1\cdots w_{s-1}\in \W_{s-1,k}$ and noticing that
\[A(x,1)^{w_0\cdots w_{s-2}}=\frac{x^{s-1}}{1-kx},\] we conclude that
\[F_k(x)=\frac{x^s}{1-kx}\sum_{w_{0}\cdots w_{s-1}\in\W_{s,k}}\max\left\{ \min\left\{ w_{0},\ldots,w_{s-1}\right\} -r+1,0\right\} +kxF_k(x).\] Solving for $F_k(x)$ and using \eqref{eq;1122}, the assertion follows.
\end{proof}

\begin{table}[H]
\begin{center}
{\renewcommand{\arraystretch}{1.1}
\begin{tabular}{||c|c| c||c|c ||}
 \hline
 $k$ & $r$ & $s$& $f_{n,k}$ & OEIS \\ [0.5ex]
 \hline\hline
 2&1&1   &$3n2^{n - 1}$ &\seqnum{A167667} \\ \hline
 2&1&2    &$5(n-1)2^{n-2}$ &- \\ \hline
 2&1&3    &$9(n-2)2^{n-3}$ &- \\ \hline

 2&2&1   &$n 2^{n-1}$ &\seqnum{A001787} \\ \hline
 2&2&2   &$(n-1)2^{n-2}$ &\seqnum{A001787} \\ \hline
 2&2&3   &$(n-2) 2^{n-3}$ &\seqnum{A001787} \\ \hline

 3&1&1    &$6 n 3^{n-1}$ &- \\ \hline
 3&1&2  &$14 (n-1) 3^{n-2}$ &- \\ \hline
 3&1&3   &$36 (n-2) 3^{n-3}$ &- \\ \hline

 3&2&1   &$n 3^n$ &\seqnum{A036290} \\ \hline
 3&2&2   &$5 (n-1) 3^{n-2}$ &- \\ \hline
 3&2&3   &$(n-2) 3^{n-1}$ &- \\ \hline

 3&3&1   &$n3^{n - 1}$ &\seqnum{A027471} \\ \hline
 3&3&2   &$(n-1) 3^{n-2}$ &\seqnum{A027471} \\ \hline
 3&3&3   &$(n - 2)3^{n - 3}$ &\seqnum{A027471} \\ \hline
\end{tabular}
\caption{The total number $f_{n,k}$ of $r\times s$ rectangles over all words of length $n$, for several small values of $k,r$, and $s$.}\label{tablea1}}
\end{center}
\end{table}

\section{Main results - Catalan words}

Recall that the set of Catalan words of length $n$, denoted by $\C_n$, is defined as
\[\C_{n}=\left\{ w_{1}\cdots w_{n}\;:\;w_{1}=1\text{ and }w_{i+1}\leq w_{i}+1,\text{ for every \ensuremath{i\in[n-1]}}\right\}.\] Denote by $p_{n}=p_{n}(t)$ (resp.\ $q_{n}=q_{n}(t)$) the distribution on $\C_{n}$ of the number of $1\times s$ (resp.\ $r\times s$, for $r\geq 2$) rectangles and let $P(x,t)$ (resp.\ $Q(x,t)$) denote the generating function of the numbers $p_{n}$ (resp.\ $q_n$). The Catalan numbers, denoted by $C_n$, satisfy 
\begin{equation}\label{eq;733}
C_0 = 1,\; C_{n+1}=\sum_{k=0}^nC_kC_{n-k} \textnormal{ for every } n\geq 0\end{equation} (e.g., \cite[(1.1)]{St}). The generating function of the Catalan numbers, denoted by $C(x)$, satisfies $C(x)=(1-\sqrt{1-4x})/2x$ and
\begin{equation}\label{eq;7333}
xC(x)^2 -C(x)+1=0.    
\end{equation} We distinguish between two cases, namely $r=1$ and $r\geq 2$.

\subsection{\texorpdfstring{$1\times s$}{} rectangles}

\begin{theorem}
We have
\begin{equation}\label{eq;gla1}
P(x,t)= \frac{\alpha(x,t)+\beta(x,t)txP(tx,t)}{1-txP(tx,t)},
\end{equation} where
\begin{align}
\alpha(x,t) &=\sum_{n=0}^{s-1}C_{n}x^{n}+x^{s-1}\sum_{i=0}^{s-2}C_{s-2-i}\sum_{n=0}^{i}C_{n}(tx)^{n-i},\nonumber\\
\beta(x,t)&=-\sum_{n=0}^{s-2}C_{n}x^{n}+\frac{1}{t^{s-1}}\sum_{n=0}^{s-2}C_{n}(tx)^{n}.\nonumber
\end{align}
\end{theorem}

\begin{proof}
For $i\in[n]$, we denote by $\C_{n}^{(i)}$ the subset of $\C_{n}$ consisting of those Catalan words with the right-most $1$ at the $i$th position. Let $p_{n}^{(i)}$ be the restriction of $p_{n}$ to $\C_{n}^{(i)}$. Clearly, for $0\leq n\leq s-1$, we have $p_n = C_n$. Thus, assume that $n\geq s$. We have
\begin{align}
p_{n}&=\sum_{i=1}^{n}p_{n}^{(i)}\nonumber
\\&=\sum_{i=1}^{n}p_{i-1}t^{\max\left\{ 0,n-\max\left\{ 1,i-s+1\right\} +1-s+1\right\} }p_{n-i}\nonumber
\\&=\sum_{i=1}^{s-1}p_{i-1}t^{n-s+1 }p_{n-i}+\sum_{i=s}^{n}p_{i-1}t^{n-i+1 }p_{n-i}.
\end{align}
Multiplying both sides of this equation by $x^n$, summing over $n\geq s$ and adding $\sum_{n=0}^{s-1}C_{n}x^n$ to both sides, with some algebra we obtain the equation
\begin{align}
P(x,t)(1-txP(tx,t))&=\sum_{n=0}^{s-1}C_{n}x^{n}+x^{s-1}\sum_{i=0}^{s-2}C_{s-2-i}\sum_{n=0}^{i}C_{n}(tx)^{n-i}\nonumber\\&+\left(-\sum_{n=0}^{s-2}C_{n}x^{n}+\frac{1}{t^{s-1}}\sum_{n=0}^{s-2}C_{n}(tx)^{n}\right)txP(tx,t),\label{eq;lba}
\end{align}
from which the statement immediately follows.
\end{proof}

\begin{remark}\label{rem;abcd}
We observe that $P(x,1)=C(x)$, as expected. Indeed, by \eqref{eq;lba} we have
\begin{align}
P(x,1)(1-xP(x,1))&=\sum_{n=0}^{s-1}C_{n}x^{n}-x\sum_{i=0}^{s-2}C_{i}x^{i}\sum_{n=0}^{s-2-i}C_{n}x^{n}\nonumber\\&=\sum_{n=0}^{s-1}C_{n}x^{n}-\sum_{n=0}^{s-2}\left(\sum_{i=0}^{n}C_{i}C_{n-i}\right)x^{n+1}\nonumber\\
&=\sum_{n=0}^{s-1}C_{n}x^{n}-\sum_{n=0}^{s-2}C_{n+1}x^{n+1}\nonumber\\&=1,\nonumber
\end{align} where in the second equality we used \eqref{eq;733}. By \eqref{eq;7333}, $P(x,1)=C(x)$.
\end{remark}

\begin{theorem}\label{eq;hs2}
For $n\geq s-1$, the total number of $1\times s$ rectangles contained in all words belonging to $\C_{n}$ is given by 
\begin{equation}\label{eq;pp1}
g_n=\frac{1}{2}\left(\sum_{i=0}^{s-2}
\frac{s-1-i}{i+1}\binom{2i}{i}\binom{2n-2i}{n-i}
-(2s-1)\binom{2n}{n}+4^n\right).
\end{equation}
\end{theorem}

\begin{proof}
Define $H(x)=\frac{\partial}{\partial x}P(x,t)_{|(x,t)=(x,1)}$ and $G(x)=\frac{\partial}{\partial t}P(x,t)_{|(x,t)=(x,1)}$. Differentiating \eqref{eq;gla1} with respect to $x$ and substituting $t=1$, after some algebra obtain
\begin{align*}
&H(x)\left((1-xC(x))^{2}-x\overbrace{\left(\sum_{n=0}^{s-1}C_{n}x^{n}-\sum_{i=0}^{s-2}C_{i}\sum_{n=i}^{s-2}C_{n-i}x^{n+1}\right)}^{=1}\right)\\
&=\overbrace{\left(\sum_{n=0}^{s-1}nC_{n}x^{n-1}-\sum_{i=0}^{s-2}C_{i}\sum_{n=i}^{s-2}(n+1)C_{n-i}x^{n}\right)}^{=0}\\
&+C(x)\overbrace{\left(\sum_{n=0}^{s-1}(1-n)C_{n}x^{n}
+\sum_{i=0}^{s-2}C_{i}\sum_{n=i}^{s-2}nC_{n-i}x^{n+1}\right)}^{=1}.
\end{align*} It follows that 
\begin{equation}\label{eq;hua1}
H(x)=\frac{C(x)}{(1-xC(x))^{2}-x}.
\end{equation} 
Now, differentiating \eqref{eq;gla1} with respect to $t$ and substituting $t=1$, after some algebra, together with \eqref{eq;hua1}, we obtain
\begin{align*}
&G(x)\left((1-xC(x))^{2}-x\right)\\
&=\left(-\sum_{i=0}^{s-2}C_{i}\sum_{n=i}^{s-2}(n-s+2)C_{n-i}x^{n+1}+C(x)\sum_{n=0}^{s-2}(n+1-s)C_{n}x^{n+1}\right)(1-xC(x))\nonumber\\&+xC(x)+\frac{x^{2}C(x)}{(1-xC(x))^{2}-x},
\end{align*} 
which, by the fact that $(1-xC(x))^{2}-x=1-x(2+C(x))$, implies
\begin{align}
&G(x)\nonumber\\
&=\frac{\left(-\sum_{i=0}^{s-2}C_{i}\sum_{n=i}^{s-2}(n-s+2)C_{n-i}x^{n+1}+C(x)\sum_{n=0}^{s-2}(n+1-s)C_{n}x^{n+1}\right)(1-xC(x))}{1-x(2+C(x))\nonumber}\\
&+\frac{xC(x)}{1-x(2+C(x))}+\frac{x^{2}C(x)}{\left(1-x(2+C(x))\right)^{2}}\label{eq;iou}.
\end{align}
Making the dependency on $s$ explicit, we set $G_s(x)=G(x)$ and define $\mathcal{G}(x,y)=\sum_{s\geq1}G_s(x)y^s$. Multiplying both sides of \eqref{eq;iou} by $y^s$ and summing over $s\geq1$, we obtain
\begin{align*}
\mathcal{G}(x,y)&=\frac{\left(\frac{xy^3}{(1-y)^2}(C(xy))^2-\frac{xy^2}{(1-y)^2}C(x)C(xy)\right)(1-xC(x))+\frac{xy}{1-y}C(x)}{1-x(2+C(x))}\\
&+\frac{x^2yC(x)}{(1-y)(1-x(2+C(x)))^2}.
\end{align*}
After some algebra, we obtain 
\begin{equation}\label{eq;iio}
\mathcal{G}(x,y)=-\frac{y\left(1+\frac{1}{\sqrt{1-4x}}\right)\sqrt{1-4xy}}{4x(1-y)^2}
-\frac{y(2xy+2x-1)}{4x(1-y)^2\sqrt{1-4x}}
-\frac{y(2xy+2x-1)}{4x(1-y)^2(1-4x)}.
\end{equation}
We have $\sqrt{1-4xy}=1-\sum_{s\geq0}2C_sx^{s+1}y^{s+1}$. Hence, by comparing the coefficient of $y^s$ on both sides of \eqref{eq;iio}, we obtain
\begin{align*}
G_s(x)&=-\frac{s\left(1+\frac{1}{\sqrt{1-4x}}\right)}{4x}
+\frac{1}{4}\left(1+\frac{1}{\sqrt{1-4x}}\right)\sum_{i=1}^{s-1}2iC_{s-1-i}x^{s-1-i}\\
&-\frac{s-1}{2\sqrt{1-4x}}
-\frac{s(2x-1)}{4x\sqrt{1-4x}}-\frac{s-1}{2(1-4x)}
-\frac{s(2x-1)}{4x(1-4x)}.
\end{align*}
We have $\frac{1}{\sqrt{1-4x}}=\sum_{n\geq0}\binom{2n}{n}x^n$. Thus, by finding the coefficient of $x^n$, the proof is complete.
\end{proof}

\begin{table}[H]
\begin{center}
{\renewcommand{\arraystretch}{1.3}
\begin{tabular}{|| c||c|c ||}
 \hline
   $s$& $g_{n}$ & OEIS \\ [0.5ex]
 \hline\hline
 1   &$\frac{1}{2}\left(4^n-\binom{2n}{n}\right)$ &\seqnum{A000346} \\ \hline
 2    &$\frac{1}{2}\left(4^n-2\binom{2n}{n}\right)$ &- \\ \hline
 3    &$\frac{1}{2}\left(\binom{2n-2}{n-1}-3\binom{2n}{n}+4^{n}\right)$ &- \\ \hline
\end{tabular}
\caption{The total number $g_{n}$ of $1\times s$ rectangles over all Catalan words of length $n$, for $s=1,2,$ and $3$.}\label{tablea2}}
\end{center}
\end{table}

\begin{remark}
The total number $g_{n}$ of $1\times s$ rectangles over all Catalan words of length $n$, for several small values of $s$ is given Table \ref{tablea2}. Notice that, for $s=1$, \eqref{eq;pp1} is exactly the total area over all Catalan words of length $n$ (see \cite[Corollary 4.4.]{6}).
\end{remark}

\subsection{\texorpdfstring{$r\times s$}{} rectangles, where \texorpdfstring{$r\geq 2$}{}}

For $i\in[n]$, we set \[\C_{n}^{(=i)}=\left\{ w_{1}\cdots w_{n}\in\C_{n}:\;w_{n}=i\right\},\] and denote by $q_{n}^{(=i)}$ the restriction of $q_{n}$ to $\C_{n}^{(=i)}$. Let $Q^{(=i)}(x,t)$ denote the corresponding generating function. By \cite[Theorem 2.2]{6}, \begin{equation}\label{7da}
\left|\C_{n}^{(=i)}\right|=\frac{i}{2n-i}\binom{2n-i}{n}.\end{equation}
It is well known (e.g., \cite[(9)]{6}) that, for every $i\in[n]$,
\begin{equation}\label{8da}
\sum_{k=i}^{n}\frac{k}{2n-k}\binom{2n-k}{n}=\frac{i+1}{2n+1-i}\binom{2n+1-i}{n+1}.   
\end{equation}
We shall also need the following combinatorial identity, for which we could not find a reference. The corresponding integer sequence is registered as \seqnum{A000245} in \cite{SL}.
\begin{lemma}\label{lem;7721}
We have \[\sum_{k=1}^{n}\frac{k^{2}}{2n-k}\binom{2n-k}{n}=C_{n+1}-C_n.\]   
\end{lemma}
\begin{proof}
We have 
\begin{align*}
\sum_{k=1}^{n}\frac{k^{2}}{2n-k}\binom{2n-k}{n}	&=\sum_{k=1}^{n}\sum_{j=1}^{k}\frac{k}{2n-k}\binom{2n-k}{n}\\
&=\sum_{j=1}^{n}\sum_{k=j}^{n}\frac{k}{2n-k}\binom{2n-k}{n}\\
&=\sum_{j=1}^{n}\frac{j+1}{2n+1-j}\binom{2n+1-j}{n+1}\\
&=\sum_{j=2}^{n+1}\frac{j}{2(n+1)-j}\binom{2(n+1)-j}{n+1}\\
&=\sum_{j=1}^{n+1}\frac{j}{2(n+1)-j}\binom{2(n+1)-j}{n+1}-\frac{1}{2(n+1)-1}\binom{2(n+1)-1}{n+1}\\
&=C_{n+1}-C_{n},    
\end{align*} where in the third and last equalities we used \eqref{8da}.
\end{proof}

We shall also need \[\C_{n}^{(i),\geq r}=
\begin{cases}
\left\{ w_{1}\cdots w_{n}\in\C_{n}:\;w_{i}=r-1\text{ and }w_{j}\geq r\text{, for every \ensuremath{i+1\leq j\leq n}}\right\} & \textnormal{if } i<n\\
\left\{ w_{1}\cdots w_{n}\in\C_{n}:\;w_{n}\leq r-1\right\}& \textnormal{if } i=n.
\end{cases}\]
Let $q_{n}^{(i),\geq r}$ denote the restriction of $q_{n}$ to $\C_{n}^{(i),\geq r}$ and let $Q^{(i),\geq r}(x,t)$ be the corresponding generating function.

Additionally, we shall make use of the Chebyshev polynomials of the second kind, denoted by $U_n(x)$. These satisfy (e.g., \cite[(1.6a) and (1.6b)]{HM}): \[U_0(x)=1,\;  U_1(x)=2x,\; \textnormal{ and } U_n(x) = 2xU_{n-1}(x)-U_{n-2}(x) \textnormal{ for every } n\geq 2.\]
For $n\geq 1$, we set $Z_n(x)=\sqrt{x^{n-1}}U_{n-1}\left(\frac{1}{2\sqrt{x}}\right)$. Thus, 
\begin{equation}\label{eq;eq;zxc}
Z_1(x)=1,\;  Z_2(x)=1,\; \textnormal{ and } Z_n(x) = Z_{n-1}(x)-xZ_{n-2}(x) \textnormal{ for every } n\geq 3.
\end{equation} It will be convenient to set $U_{-1}(x) =Z_0(x)=0$.

In the proof of Theorem \ref{7daad} we shall encounter a certain linear system of equations. In the following lemma we establish its unique solution.

\begin{lemma}\label{lem;ha1}
Let $i,j\in[r-1]$ and let $x,y$ be two indeterminates. Let $M$ be the square matrix of size $r-1$ given by
\[(M)_{ij}  = 
\begin{cases}
1-x&\textnormal{if } j<r-1 \textnormal{ and } i = j\\    
-x&\textnormal{if } j < r-1 \textnormal{ and } j>i \textnormal{ or } j = i-1 \\
-xy&\textnormal{if } j=r-1 \textnormal{ and } i<r-1 \\
1-xy&\textnormal{if } j=r-1 \textnormal{ and }i = r-1\\
0&\textnormal{otherwise}.
\end{cases}\]
Thus,
\[M =
\begin{pmatrix}1-x & -x & \cdots & -x & -xy\\
-x & 1-x & \ddots & -x & -xy\\
0 & \ddots & \ddots & \ddots & \vdots\\
\vdots & \ddots & -x & 1-x & -xy\\
0 & \cdots & 0 & -x & 1-xy
\end{pmatrix}.\]
Let $\beta=(\beta_1,\ldots,\beta_{r-1})^T\in\R^{r-1}$. Then the unique solution $c=(c_1,\ldots,c_{r-1})^T\in\R^{r-1}$ of the equation $Mc=\beta$ is given by 
\[c_i=\begin{cases}
\frac{Z_{i+1}(x)}{Z_{i+2}(x)}\gamma_{i}+\sum_{j=i+1}^{r-2}\frac{xZ_{i}(x)}{Z_{j+2}(x)}\gamma_{j}+\frac{xyZ_{i}(x)}{Z_{r}(x)-xyZ_{r-1}(x)}\gamma_{r-1} & \textnormal{if }i<r-1\\
\frac{Z_{r}(x)}{Z_{r}(x)-xyZ_{r-1}(x)}\gamma_{r-1} & \text{if }i=r-1,
\end{cases}\] where 
\begin{align}
\gamma_{i}&=\frac{x^{i}}{Z_{i+1}(x)}\sum_{j=1}^{i}\frac{Z_{j+1}(x)}{x^{j}}\beta_{j}.\nonumber
\end{align}
\end{lemma}

\begin{proof}
It is easy to verify that the $LU$ decomposition of $M$ is given by  
\begin{align}
(L)_{ij}&=\begin{cases} 
1 &\textnormal{if } i = j\\ -\frac{xZ_{i}(x)}{Z_{i+1}(x)}&\textnormal{if } i = j+1\\
0 &\textnormal{otherwise},\\    
\end{cases}\nonumber\\
(U)_{ij}&=\begin{cases} 
\frac{Z_{i+2}(x)}{Z_{i+1}(x)}&\textnormal{if } i< r-1 \textnormal{ and } j = i\\
-\frac{xZ_{i}(x)}{Z_{i+1}(x)}&\textnormal{if } i< r-1 \textnormal{ and } i< j < r-1\\
-\frac{xyZ_{i}(x)}{Z_{i+1}(x)}&\textnormal{if } i< r-1 \textnormal{ and } j=r-1\\
1-\frac{xyZ_{r-1}(x)}{Z_{r}(x)}&\textnormal{if } i= r-1 \textnormal{ and } j=i\\
0 &\textnormal{otherwise}.\\    
\end{cases}\nonumber
\end{align} Inverting $L$ and $U$, we obtain 
\begin{align}
(L^{-1})_{ij}&=\begin{cases} 
\frac{x^iZ_{j+1}(x)}{x^jZ_{i+1}(x)}&\textnormal{if } i \geq j\\
0 &\textnormal{otherwise},\\    
\end{cases}\nonumber\\
(U^{-1})_{ij}&=\begin{cases} 
\frac{Z_{i+1}(x)}{Z_{i+2}(x)}&\textnormal{if } i< r-1 \textnormal{ and } j = i\\
\frac{xZ_{i}(x)}{Z_{j+2}(x)}&\textnormal{if } i< r-1 \textnormal{ and } i< j < r-1\\
\frac{xyZ_{i}(x)}{Z_{r}(x)-xyZ_{r-1}(x)}&\textnormal{if } i< r-1 \textnormal{ and } j=r-1\\
\frac{Z_{r}(x)}{Z_{r}(x)-xyZ_{r-1}(x)}&\textnormal{if } i= r-1 \textnormal{ and } j=i\\
0 &\textnormal{otherwise}.\\    
\end{cases}\nonumber
\end{align} Evaluating $c=U^{-1}(L^{-1}\beta)$, the assertion follows.
\end{proof}

\begin{theorem}\label{7daad}
We have
\begin{equation}\label{eq;842p}
Q(x,t) = \frac{Z_{r-1}(x)-xZ_{r-2}(x)P(x,t)}{Z_{r}(x)-xZ_{r-1}(x)P(x,t)}.
\end{equation}
\end{theorem}

\begin{proof}
Assume that $n\geq r$. We claim that the following relations hold:
\begin{equation} \label{eq;74g}
q_{n}^{(i),\geq r}=
\begin{cases}
q_{i}^{(=r-1)}p_{n-i}&\textnormal{if } r-1\leq i\leq n-1\\
\sum_{j=1}^{r-1}q_{n}^{(=j)}& \textnormal{if } i= n.
\end{cases}
\end{equation}  Furthermore,
\begin{equation}\label{eq;kht}
q_{n}^{(=i)}=
\begin{cases}
\sum_{k=1}^{r-1}q_{n-1}^{(=k)}+\sum_{k=r-1}^{n-2}q_{n-1}^{(k),\geq r}&\textnormal{if } i=1\\
\sum_{k=i-1}^{r-1}q_{n-1}^{(=k)}+\sum_{k=r-1}^{n-2}q_{n-1}^{(k),\geq r}&\textnormal{if } 2\leq i\leq r-1.
\end{cases}
\end{equation}
We have
\begin{align}
q_{n}&=\sum_{i=r-1}^{n}q_{n}^{(i),\geq r}\nonumber\\&=\sum_{i=r-1}^{n-1}q_{i}^{(=r-1)}p_{n-i}+\sum_{i=1}^{r-1}q_{n}^{(=i)}    \nonumber\\
&=\sum_{i=r-1}^{n-1}q_{i}^{(=r-1)}p_{n-i}+\sum_{k=1}^{r-2}(k+1)q_{n-1}^{(=k)}+(r-1)\sum_{i=r-1}^{n-1}q_{i}^{(=r-1)}p_{n-1-i}, \label{eq;8ga}
\end{align}
where we substituted $q_{n}^{(i),\geq r}$ with \eqref{eq;74g} and $q_{n}^{(=i)}$ with \eqref{eq;kht}. Multiplying both sides of \eqref{eq;8ga} by $x^n$, summing over $n\geq r$ and adding $\sum_{n=0}^{r-1}C_{n}x^n$ to both sides, with some algebra and \eqref{7da} we obtain \begin{align}
Q(x,t)&=\sum_{n=0}^{r-1}C_{n}x^{n}+Q^{(=r-1)}(x,t)(P(x,t)((r-1)x+1)-1)\nonumber\\&+x\sum_{k=1}^{r-2}(k+1)\left(Q^{(=k)}(x,t)-\sum_{n=k}^{r-2}\frac{k}{2n-k}\binom{2n-k}{n}x^{n}\right).  \label{eq;dba}
\end{align} 
Considering some of the terms on the right-hand side of \eqref{eq;dba}, we notice that, for $r\geq 3$, we have
\begin{align}
&\sum_{n=0}^{r-1}C_{n}x^{n}-\sum_{k=1}^{r-2}\sum_{n=k}^{r-2}\frac{k^{2}}{2n-k}\binom{2n-k}{n}x^{n+1}-\sum_{k=1}^{r-2}\sum_{n=k}^{r-2}\frac{k}{2n-k}\binom{2n-k}{n}x^{n+1}\nonumber\\&=\sum_{n=0}^{r-1}C_{n}x^{n}-\sum_{n=1}^{r-2}\overbrace{\left(\sum_{k=1}^{n}\frac{k^{2}}{2n-k}\binom{2n-k}{n}\right)}^{=C_{n+1}-C_{n}, \textnormal{ by Lemma \ref{lem;7721}}}x^{n+1}-\sum_{n=1}^{r-2}\overbrace{\left(\sum_{k=1}^{n}\frac{k}{2n-k}\binom{2n-k}{n}\right)}^{=C_n, \textnormal{ by \eqref{8da}}}x^{n+1}    \nonumber\\
&=\sum_{n=0}^{r-1}C_{n}x^{n}-\sum_{n=1}^{r-2}(C_{n+1}-C_{n})x^{n+1}-\sum_{n=1}^{r-2}C_{n}x^{n+1}\nonumber\\&=1+x.\nonumber
\end{align}
Thus, for $r\geq 2$, we have \begin{equation}\label{eq;h81}
Q(x,t)=1+x+Q^{(=r-1)}(x,t)(P(x,t)((r-1)x+1)-1)+x\sum_{k=1}^{r-2}(k+1)Q^{(=k)}(x,t).\end{equation} Now, substituting \eqref{eq;74g} in  \eqref{eq;kht}, for $i\in[r-1]$, we obtain
\begin{equation}\label{eq;6jc}
q_{n}^{(=i)}=\begin{cases}
\sum_{k=1}^{r-1}q_{n-1}^{(=k)}+\sum_{k=r-1}^{n-2}q_{k}^{(=r-1)}p_{n-1-k} & \textnormal{if } i=1\\
\sum_{k=i-1}^{r-1}q_{n-1}^{(=k)}+\sum_{k=r-1}^{n-2}q_{k}^{(=r-1)}p_{n-1-k} & \textnormal{if } 2\leq i\leq r-1.
\end{cases}
\end{equation}
Multiplying both sides of \eqref{eq;6jc} by $x^n$, summing over $n\geq r+1$ and adding $\sum_{n=i}^{r}q_{n}^{(=i)}x^{n}$ to both sides, with some algebra we obtain
\begin{align}
&Q^{(=1)}(x,t)=\nonumber\\
&
\sum_{n=1}^{r}q_{n}^{(=1)}x^{n}+x\sum_{n=1}^{r-1}Q^{(=n)}(x,t)-x\sum_{k=1}^{r-1}\sum_{n=k}^{r-1}q_{n}^{(=k)}x^{n}+xQ^{(=r-1)}(x,t)(P(x,t)-1) \label{eq;kk1}
\end{align}
and, for $2\leq i\leq r-1$,
\begin{align}
&Q^{(=i)}(x,t)=\nonumber\\
&
\sum_{n=i}^{r}q_{n}^{(=i)}x^{n}+x\sum_{n=i-1}^{r-1}Q^{(=n)}(x,t)-x\sum_{k=i-1}^{r-1}\sum_{n=k}^{r-1}q_{n}^{(=k)}x^{n}+xQ^{(=r-1)}(x,t)(P(x,t)-1) \label{eq;kk2}
\end{align}
Set 
\[
\beta_i=\begin{cases}
\sum_{n=1}^{r}q_{n}^{(=1)}x^{n}-x\sum_{k=1}^{r-1}\sum_{n=k}^{r-1}q_{n}^{(=k)}x^{n}&\textnormal{if } i =1\\ \sum_{n=i}^{r}q_{n}^{(=i)}x^{n}-x\sum_{k=i-1}^{r-1}\sum_{n=k}^{r-1}q_{n}^{(=k)}x^{n}&  \textnormal{if } 2\leq i \leq r-1.
\end{cases}\]
Using \eqref{7da} and \eqref{8da}, we have
\begin{align}
\beta_i&=\begin{cases}
\sum_{n=1}^{r}\frac{1}{2n-1}\binom{2n-1}{n}x^{n}-x\sum_{k=1}^{r-1}\sum_{n=k}^{r-1}\frac{k}{2n-k}\binom{2n-k}{n}x^{n} & \textnormal{if }i=1\\
\sum_{n=i}^{r}\frac{i}{2n-i}\binom{2n-i}{n}x^{n}-x\sum_{k=i-1}^{r-1}\sum_{n=k}^{r-1}\frac{k}{2n-k}\binom{2n-k}{n}x^{n} & \textnormal{if }2\leq i\leq r-1.
\end{cases}\nonumber\\
&=\begin{cases}
x+\sum_{n=1}^{r-1}\left(\frac{1}{2n+1}\binom{2n+1}{n+1}-\sum_{k=1}^{n}\frac{k}{2n-k}\binom{2n-k}{n}\right)x^{n+1} & \textnormal{if }i=1\\
\sum_{n=i-1}^{r-1}\left(\frac{i}{2n+2-i}\binom{2n+2-i}{n+1}-\sum_{k=i-1}^{n}\frac{k}{2n-k}\binom{2n-k}{n}\right)x^{n+1} & \textnormal{if }2\leq i\leq r-1
\end{cases}\nonumber\\
&=\begin{cases}
x & \textnormal{if }i=1\\
0 & \textnormal{if }2\leq i\leq r-1.
\end{cases}\nonumber
\end{align}
We obtain the following system of $r-1$ equations in the $r-1$ indeterminates $Q^{(=i)}(x,t), i\in[r-1]$:
\[\begin{cases}
Q^{(=1)}(x,t)-\sum_{n=1}^{r-2}xQ^{(=n)}-xP(x,t)Q^{(=r-1)}(x,t)&=\beta_{1}    \\
\vdots&\vdots\\
Q^{(=i)}(x,t)-\sum_{n=i-1}^{r-2}xQ^{(=n)}(x,t)-xP(x,t)Q^{(=r-1)}(x,t)&=\beta_{i}\\
\vdots&\vdots\\
Q^{(=r-1)}(x,t)-\sum_{n=r-2}^{r-2}xQ^{(=n)}(x,t)-xP(x,t)Q^{(=r-1)}(x,t)&=\beta_{r-1}.
\end{cases}\] By Lemma \ref{lem;ha1}, 
\[Q^{(=i)}(x,t)=\begin{cases}
\frac{Z_{i+1}(x)}{Z_{i+2}(x)}\gamma_{i}+\sum_{j=i+1}^{r-2}\frac{xZ_{i}(x)}{Z_{j+2}(x)}\gamma_{j}+\frac{xP(x,t)Z_{i}(x)}{Z_{r}(x)-xP(x,t)Z_{r-1}(x)}\gamma_{r-1} & \textnormal{if }i<r-1\\
\frac{Z_{r}(x)}{Z_{r}(x)-xP(x,t)Z_{r-1}(x)}\gamma_{r-1} & \text{if }i=r-1,
\end{cases}\] where 
\[\gamma_{i}=\frac{x^{i}}{Z_{i+1}(x)}\sum_{j=1}^{i}\frac{Z_{j+1}(x)}{x^{j}}\beta_{j}=\frac{x^{i}Z_{2}(x)}{Z_{i+1}(x)}=\frac{x^{i}}{Z_{i+1}(x)}.\] Thus, 
\[Q^{(=i)}(x,t)=\begin{cases}
\frac{x^{i}}{Z_{i+2}(x)}+\sum_{j=i+1}^{r-2}\frac{x^{j+1}Z_{i}(x)}{Z_{j+1}(x)Z_{j+2}(x)}+\frac{x^{r}P(x,t)Z_{i}(x)}{Z_{r}(x)(Z_{r}(x)-xP(x,t)Z_{r-1}(x))} & \textnormal{if }i<r-1\\
\frac{x^{r-1}}{Z_{r}(x)-xP(x,t)Z_{r-1}(x)} & \text{if }i=r-1.
\end{cases}\]
Substituting this into \eqref{eq;h81}, we distinguish between two cases: If $r=2$ then \[Q(x,t)=\frac{1}{1-xP(x,t)}=\frac{Z_{1}(x)-xZ_{0}(x)P(x,t)}{Z_{2}(x)-xZ_{1}(x)P(x,t)}\] and the assertion holds true. Thus, we assume that $r\geq 3$. We have
\begin{align}
Q(x,t)&=1+x+\frac{x^{r-1}(P(x,t)((r-1)x+1)-1)}{Z_{r}(x)-xP(x,t)Z_{r-1}(x)}+x\sum_{k=1}^{r-2}(k+1)\bigg(\frac{x^{k}}{Z_{k+2}(x)}+\nonumber\\&\hspace{3cm}\sum_{j=k+1}^{r-2}\frac{x^{j+1}Z_{k}(x)}{Z_{j+1}(x)Z_{j+2}(x)}+\frac{x^{r}P(x,t)Z_{k}(x)}{Z_{r}(x)(Z_{r}(x)-xP(x,t)Z_{r-1}(x))}\bigg)\label{eq;780}
\end{align} Let us set
\begin{align*}
A_r&=x(1+x)Z_r(x)-x^{r-1}+xZ_r(x)\sum_{k=1}^{r-2}\left(\frac{(k+1)x^k}{Z_{k+2}(x)}+\sum_{j=k+1}^{r-2}\frac{(k+1)x^{j+1}Z_k(x)}{Z_{j+1}(x)Z_{j+2}(x)}\right),\\
B_r&=(1+x)Z_{r-1}(x)-x^{r-2}((r-1)x+1)+x^r\sum_{k=1}^{r-2}\frac{(k+1)Z_k(x)}{Z_r(x)}-Z_{r-1}(x)A_r.
\end{align*} With this, we may write \eqref{eq;780} as
\begin{equation}\label{eq;mmk}
(Z_r(x)-xP(x,t)Z_{r-1}(x))Q(x,t)=A_r-xP(x,t)B_r.
\end{equation} Using the recursion for $Z_n$ stated in \eqref{eq;eq;zxc}, the following three identities are easily proved by induction on $r$:
\begin{align*}
&x^2\sum_{k=1}^r(k+1)Z_k(x)=1- (r+1)xZ_{r+2}(x)-Z_{r+3}(x),\\
&\sum_{j=k+1}^r\frac{x^{j-k-1}}{Z_{j+1}(x)Z_{j+2}(x)}=\frac{Z_{r-k}(x)}{Z_{k+2}(x)Z_{r+2}(x)},\\
&\sum_{k=1}^{r-2}\frac{(k+1)x^{k+1}(Z_r(x)+x^2Z_k(x)Z_{r-k-2}(x))}{Z_{k+2}(x)}=Z_{r-1}(x)-(1+x)Z_r(x)+x^{r-1}.
\end{align*}
Hence, $A_r$ can be simplified as
\begin{align*}
A_r&=(1+x)Z_r(x)-x^{r-1}+xZ_r(x)\sum_{k=1}^{r-2}\frac{(k+1)x^k}{Z_{k+2}(x)}
+xZ_r(x)\sum_{k=1}^{r-3}\sum_{j=k+1}^{r-2}\frac{(k+1)x^{j+1}Z_k(x)}{Z_{j+1}(x)Z_{j+2}(x)},\\
&=(1+x)Z_r(x)+(r-2)x^{r-1}+\sum_{k=1}^{r-3}\frac{(k+1)x^{k+1}(Z_r(x)+x^2Z_k(x)Z_{r-k-2}(x))}{Z_{k+2}(x)}\\
&=(1+x)Z_r(x)-x^{r-1}+\sum_{k=1}^{r-2}\frac{(k+1)x^{k+1}(Z_r(x)+x^2Z_k(x)Z_{r-k-2}(x))}{Z_{k+2}(x)}\\
&=(1+x)Z_r(x)-x^{r-1}+Z_{r-1}(x)-(1+x)Z_r(x)+x^{r-1}\\
&=Z_{r-1}(x).
\end{align*}
Similarly, $B_r(x)$ can be simplified as
\begin{align*}
B_r&=(1+x)Z_{r-1}(x)-x^{r-2}((r-1)x+1)+xZ_{r-1}(x)\sum_{k=1}^{r-2}\frac{(k+1)x^k}{Z_{k+2}(x)}\\
&+xZ_{r-1}(x)\sum_{k=1}^{r-3}(k+1)x^{k+2}Z_k(x)\sum_{j=k+1}^{r-2}\frac{x^{j-k-1}}{Z_{j+1}(x)Z_{j+2}(x)}
-\frac{x^r}{Z_r(x)}\sum_{k=1}^{r-2}(k+1)Z_k(x)\\
&=(1+x)Z_{r-1}(x)-x^{r-2}((r-1)x+1)+\frac{(r-1)x^{r-1}Z_{r-1}(x)}{Z_r(x)}
+xZ_{r-1}(x)\sum_{k=1}^{r-3}\frac{(k+1)x^k}{Z_{k+2}(x)}\\
&+xZ_{r-1}(x)\sum_{k=1}^{r-3}\frac{(k+1)x^{k+2}Z_k(x)Z_{r-k-2}(x)}{Z_{k+2}(x)Z_r(x)}-\frac{x^r}{Z_r(x)}\sum_{k=1}^{r-2}(k+1)Z_k(x)\\
&=(1+x)Z_{r-1}(x)-x^{r-2}+\frac{(r-1)x^{r-1}Z_{r-1}(x)}{Z_r(x)}\\
&+\frac{Z_{r-1}(x)}{Z_r(x)}\sum_{k=1}^{r-3}\frac{(k+1)x^{k+1}(Z_r(x)+x^2Z_k(x)Z_{r-k-2}(x))}{Z_{k+2}(x)}-\frac{x^{r-2}(1-Z_{r+1}(x))}{Z_r(x)}\\
&=(1+x)Z_{r-1}(x)-x^{r-2}+\frac{Z_{r-1}(x)(Z_{r-1}(x)-(1+x)Z_r(x)+x^{r-1})}{Z_r(x)}-\frac{x^{r-2}(1-Z_{r+1}(x))}{Z_r(x)}\\
&=\frac{Z_{r-1}^2(x)-x^{r-2}}{Z_r(x)}\\
&=\frac{Z_{r-2}Z_r(x)+x^{r-2}-x^{r-2}}{Z_r(x)}\\
&=Z_{r-2}(x).
\end{align*}
Substituting these into \eqref{eq;mmk}, we have
\begin{align*}
(Z_r(x)-xP(x,t)Z_{r-1}(x))Q(x,t)=Z_{r-1}(x)-xP(x,t)Z_{r-2}(x),
\end{align*}
which completes the proof.
\end{proof}

\begin{theorem}\label{eq;hs22}
The generating function for the total number $h_n$ of $r\times s$ rectangles over all Catalan words of length $n$ is given by 
\[\left(\frac{\partial}{\partial t}P(x,t)_{|(x,t)=(x,1)}\right)(C(x)-1)^{r-1}.\] In particular,
\begin{align}
h_n&=\sum_{k=r}^{n+1-s}\frac{r-1}{k+r-2}\binom{2k-3}{k-r}\bigg(\sum_{i=0}^{s-2}\frac{s-1-i}{i+1}\binom{2i}{i}\binom{2(n+1-k)-2i}{n+1-k-i}-\nonumber\\&\hspace{8cm} (2s-1)\binom{2(n+1-k)}{n+1-k}+4^{n+1}\bigg).\label{r22}
\end{align}
\end{theorem}

\begin{proof}
Using the recursions for $C(x)$ and $Z_n(x)$ stated in \eqref{eq;7333} and \eqref{eq;eq;zxc}, respectively, it is straightforward to verify that
\begin{align}
\frac{Z_{r-1}(x)-xZ_{r-2}(x)C(x)}{Z_{r}(x)-xZ_{r-1}(x)C(x)}&=C(x),\nonumber\\
-\frac{Z_{r-2}(x)-Z_{r-1}(x)C(x)}{Z_{r}(x)-xZ_{r-1}(x)C(x)}&=x^{r-2}C(x)^{2r-2}.\nonumber
\end{align}
Differentiating \eqref{eq;842p} with respect to $t$ and substituting $t=1$, we obtain 
\begin{align}
&\frac{\partial}{\partial t}Q(x,t)_{|(x,t)=(x,1)}\nonumber\\&=x\left(\frac{\partial}{\partial t}P(x,t)_{|(x,t)=(x,1)}\right)\bigg(-\frac{Z_{r-2}(x)}{Z_{r}(x)-xZ_{r-1}(x)C(x)}+\nonumber\\&\hspace{6.5cm} \frac{Z_{r-1}(x)}{Z_{r}(x)-xZ_{r-1}(x)C(x)}\frac{Z_{r-1}(x)-xZ_{r-2}(x)C(x)}{Z_{r}(x)-xZ_{r-1}(x)C(x)}\bigg)\nonumber\\
&=x\left(\frac{\partial}{\partial t}P(x,t)_{|(x,t)=(x,1)}\right)\left(-\frac{Z_{r-2}(x)-Z_{r-1}(x)C(x)}{Z_{r}(x)-xZ_{r-1}(x)C(x)}\right)\nonumber\\
&=\left(\frac{\partial}{\partial t}P(x,t)_{|(x,t)=(x,1)}\right)x^{r-1}C(x)^{2r-2}\nonumber\\
&=\left(\frac{\partial}{\partial t}P(x,t)_{|(x,t)=(x,1)}\right)(C(x)-1)^{r-1}.\nonumber
\end{align}
Now, \[(C(x)-1)^{r-1} = 2(r-1)x^r\sum_{n\geq 0}\frac{1}{n+2(r-1)}\binom{2n+2(r-1)-1}{n}x^n.\] From this, together with \eqref{eq;pp1}, the proof is complete.
\end{proof}

\begin{example}
For $r=2$ and $s=1$,  \eqref{r22} gives us $h_n =\binom{2n+1}{n}-\binom{2n+3}{n+1}+2\cdot4^{n}$, corresponding to \seqnum{A006419} in \cite{SL}.
\end{example}

\begin{conclusion*}
In this work we found the bivariate generating function for the distribution on words (resp.\ Catalan words) of the number of $r\times s$ rectangles and the generating function for their total number over all words (resp.\ Catalan words). We do not know how to solve the corresponding distribution problem for permutations. Nevertheless, we can show that the total number of $r\times s$ rectangles over all permutations of $[n]$ is given by $(n+1)!\binom{n-r+2}{s+1}/\binom{n+1}{s}$.
\end{conclusion*}

\end{document}